\pdfoutput=1
\documentclass[11pt,reqno]{amsart}
\usepackage[paperheight=279mm,paperwidth=18cm,textheight=26cm,textwidth=14cm,includehead]{geometry}
\usepackage[T1]{fontenc}
\usepackage[utf8]{inputenc}
\usepackage[unicode]{hyperref}
\hypersetup{
bookmarks=true,
colorlinks=true,
citecolor=[rgb]{0,0,0.5},
linkcolor=[rgb]{0,0,0.5},
urlcolor=[rgb]{0,0,0.75},
pdfpagemode=UseNone,
pdfstartview=FitH,
pdfdisplaydoctitle=true,
pdftitle={Variational estimates for martingale paraproducts},
pdfauthor={Kova\v{c}, Zorin-Kranich},
pdflang=en-US
}

\usepackage[style=alphabetic,maxalphanames=4]{biblatex}
\usepackage{amssymb}
\usepackage{amsmath}
\usepackage{amsthm}
\usepackage{amsfonts}
\usepackage{mathtools}
\usepackage{enumerate}

\newtheorem{theorem}{Theorem}[section]
\newtheorem{lemma}[theorem]{Lemma}
\newtheorem{corollary}[theorem]{Corollary}
\newtheorem{proposition}[theorem]{Proposition}

\numberwithin{equation}{section}
\def\LL{\textup{L}}
\newcommand{\dif}{\mathop{}\!\mathrm{d}} 
\DeclarePairedDelimiter\abs{\lvert}{\rvert}

\DeclarePairedDelimiter\norm{\lVert}{\rVert}

\providecommand\given{}
\newcommand\SetSymbol[1][]{%
\nonscript\:#1\vert
\allowbreak
\nonscript\:
\mathopen{}}
\DeclarePairedDelimiterX\Set[1]\{\}{\renewcommand\given{\SetSymbol[\delimsize]}#1}
\DeclarePairedDelimiterXPP\EE[1]{\mathbb{E}}{\lparen}{\rparen}{}{\renewcommand\given{\SetSymbol[\delimsize]}#1} 

\newcommand{\calF}{\mathcal{F}}
\newcommand\cadlag{c\`{a}dl\`{a}g}

\begin{document}

\title[Variational estimates for martingale paraproducts]{Variational estimates\\ for martingale paraproducts}

\author[V. Kova\v{c}]{Vjekoslav Kova\v{c}}
\address[Vjekoslav Kova\v{c}]{Department of Mathematics, Faculty of Science, University of Zagreb, Bijeni\v{c}ka cesta 30, 10000 Zagreb, Croatia}
\email{vjekovac@math.hr}

\author[P. Zorin-Kranich]{Pavel Zorin-Kranich}
\address[Pavel Zorin-Kranich]{Mathematical Institute, University of Bonn, Endenicher Allee 60, D-53115 Bonn, Germany}
\email{pzorin@math.uni-bonn.de}

\date{\today}


\begin{abstract}
We show that bilinear variational estimates of Do, Muscalu, and Thiele \cite{MR2949622} remain valid for a pair of general martingales with respect to the same filtration.
Our result can also be viewed as an off-diagonal generalization of the Burkholder--Davis--Gundy inequality for martingale rough paths by Chevyrev and Friz \cite{MR3909973}.
\end{abstract}

\maketitle

\section{Introduction}
If $f=(f_n)_{n=0}^{\infty}$ is a discrete-time real-valued martingale with respect to a filtration $\calF=(\calF_n)_{n=0}^{\infty}$, then \emph{L\'{e}pingle's variational inequality} \cite{MR0420837} claims
\begin{equation}\label{eq:1lepingle}
\norm[\Bigg]{ \sup_{\substack{m\in\mathbb{N}\\ n_0<n_1<\cdots<n_m}} \bigg( \sum_{k=1}^{m} \abs{f_{n_{k}}-f_{n_{k-1}}}^\varrho \bigg)^{1/\varrho} }_{\LL^p}
\leq C_{p,\varrho} \norm{f}_{\LL^p}
\end{equation}
for any exponents $p\in(1,\infty)$ and $\varrho\in(2,\infty)$. Here for any random variable $h$ we write $\norm{h}_{\LL^p}:=(\mathbb{E}\abs{h}^p)^{1/p}$ and for a martingale $f$ we set
\begin{equation}\label{eq:martLpnorm}
\norm{f}_{\LL^p}:=\sup_{n\geq0}\norm{f_n}_{\LL^p}.
\end{equation}
Inequality \eqref{eq:1lepingle} fails at the endpoint $\varrho=2$ and then the corresponding result is \emph{Bourgain's jump inequality} \cite{MR1019960},
\begin{equation}\label{eq:1jump}
\norm[\Bigg]{ \bigg( \sup_{\substack{m\in\mathbb{N}\\ n_0<n_1<\cdots<n_m}} \operatorname{card} \Set[\Big]{ k\in\Set{1,\ldots,m} \given \abs[\big]{f_{n_{k}}-f_{n_{k-1}}} \geq \lambda } \bigg)^{1/2} }_{\LL^p}
\leq C_p \lambda^{-1} \norm{f}_{\LL^p}
\end{equation}
for any exponent $p\in(1,\infty)$ and any threshold $\lambda\in(0,\infty)$.
The supremum of cardinalities on the left hand side of \eqref{eq:1jump} is usually called the \emph{$\lambda$-jump counting function} and denoted by $N_{\lambda}(f)$, so that \eqref{eq:1jump} can be rewritten more elegantly as
\[ \norm[\big]{N_{\lambda}(f)^{1/2}}_{\LL^p} \leq C_p \lambda^{-1} \norm{f}_{\LL^p}. \]
Inequalities \eqref{eq:1lepingle} and \eqref{eq:1jump} provide quantitative refinements of the martingale convergence theorem, at least for martingales that are bounded in the $\LL^p$-norm. The reader can consult the paper \cite{arxiv:1808.04592} for their elegant proofs and Banach-space-valued generalizations.

Let us now take two martingales, $f=(f_n)_{n=0}^{\infty}$ and $g=(g_n)_{n=0}^{\infty}$, with respect to the same filtration $\calF$. Their \emph{paraproduct} is a stochastic process $\Pi(f,g)=(\Pi_n(f,g))_{n=0}^{\infty}$ defined as
\[ \Pi_0(f,g) := 0, \quad \Pi_n(f,g) := \sum_{j=1}^{n} f_{j-1} (g_j - g_{j-1}) \quad \text{for}\ n\geq 1. \]
\emph{Truncated paraproducts} are random variables $\Pi_{n,n'}(f,g)$ defined for $0\leq n< n'$ as
\begin{align*}
\Pi_{n,n'}(f,g) & := \sum_{j=n+1}^{n'} (f_{j-1}-f_{n}) (g_j - g_{j-1}) \\
& = \Pi_{n'}(f,g) - \Pi_n(f,g) - f_n(g_{n'}-g_n).
\end{align*}
If
\[ \dif f_n := f_n - f_{n-1} \quad \text{for}\ n\geq 1 \]
denotes martingale differences, then the truncated paraproducts can be written, quite elegantly, as
\begin{equation}\label{eq:pprodaltexp}
\Pi_{n,n'}(f,g) = \sum_{n<i<j\leq n'} \dif f_i \dif g_j.
\end{equation}

Note that $g\mapsto\Pi(f,g)$ can be seen as a particular case of Burkholder's martingale transform \cite{MR0208647}.
He took $f$ to be an arbitrary process adapted to the filtration $\calF$ and bounded in the $\LL^\infty$-norm, $\norm{h}_{\LL^\infty} := \operatorname{ess\,sup}\abs{h}$, and showed
\[ \norm{\Pi(f,g)}_{\LL^q} \leq C_{q} \norm{f}_{\LL^\infty} \norm{g}_{\LL^q} \]
for any $q\in(1,\infty)$. On the other hand, we additionally assume that $f$ is a martingale, possibly unbounded.
Indeed, the word ``paraproduct'' will be preferred because martingales $f$ and $g$ can be treated symmetrically thanks to a simple summation by parts identity.
Estimates for martingale paraproducts outside Burkholder's range were first studied by Ba\~{n}uelos and Bennett \cite{MR955015} (even though in the continuous time and with respect to the Brownian filtration only) and by Chao and Long \cite{MR1079887}.
Inequalities on the $\LL^p$-spaces in the largest possible open range of exponents follow from \cite[Theorem~7]{MR1079887}:
\begin{equation}\label{eq:estChaoLong}
\norm{\Pi(f,g)}_{\LL^r} \leq C_{p,q} \norm{f}_{\LL^p} \norm{g}_{\LL^q},
\end{equation}
whenever
\begin{equation}\label{eq:exponents}
p,q\in(1,\infty),\ r\in\Big(\frac{1}{2},\infty\Big),\ \frac{1}{p}+\frac{1}{q}=\frac{1}{r}.
\end{equation}
Indeed, \cite[Theorem~7]{MR1079887} deals with a maximal estimate, namely
\begin{equation}\label{eq:estChaoLongmax}
\norm[\Big]{\sup_{n\geq0}|\Pi_n(f,g)|}_{\LL^r} \leq C_{p,q} \norm{f}_{\LL^p} \norm{g}_{\LL^q},
\end{equation}
which is stronger than \eqref{eq:estChaoLong} when $r\leq1$.
If $r>1$, then $\Pi(f,g)$ is again a martingale adapted to $\calF$, so in particular it also satisfies \eqref{eq:1lepingle}.
However, one still cannot relax the condition $\varrho>2$ for general $f$ and $g$.

It is a bit surprising that there exists a variant of L\'{e}pingle's inequality for truncated martingale paraproducts that allows $\varrho$ to go below $2$. This is the main result of our paper and it is a generalization of Theorem 1.2 from the paper \cite{MR2949622} by Do, Muscalu, and Thiele.

\begin{theorem}\label{thm:main}
Take exponents $p,q,r$ satisfying \eqref{eq:exponents}. If $f=(f_n)_{n=0}^{\infty}$ and $g=(g_n)_{n=0}^{\infty}$ are martingales with respect to the same filtration, then
\begin{equation}\label{eq:2variation}
\norm[\Bigg]{ \sup_{\substack{m\in\mathbb{N}\\ n_0<n_1<\cdots<n_m}} \bigg( \sum_{k=1}^{m} \abs[\Big]{ \Pi_{n_{k-1}, n_{k}}(f,g) }^\varrho \bigg)^{1/\varrho} }_{\LL^r}
\leq C_{p,q,\varrho} \norm{f}_{\LL^p} \norm{g}_{\LL^q}
\end{equation}
for any $\varrho\in(1,\infty)$ and
\begin{multline}\label{eq:2jump}
\norm[\Bigg]{ \sup_{\substack{m\in\mathbb{N}\\ n_0<n_1<\cdots<n_m}} \operatorname{card} \Set[\bigg]{ k\in\Set{1,\ldots,m} \given \abs[\Big]{ \Pi_{n_{k-1}, n_{k}}(f, g) } \geq \lambda } }_{\LL^r} \\
\leq C_{p,q} \lambda^{-1} \norm{f}_{\LL^p} \norm{g}_{\LL^q}
\end{multline}
for any $\lambda\in(0,\infty)$.
\end{theorem}

Indeed, Do, Muscalu, and Thiele \cite{MR2949622} considered variants of Theorem~\ref{thm:main} for either dyadic martingales or Littlewood--Paley-type convolutions. They motivate their result by an application to the bilinear iterated Fourier integral in the sequel paper \cite{MR3596720}.
The main purpose of this note is to generalize their result to arbitrary martingales, since \cite{MR2949622} repeatedly relies on doubling conditions to both raise and lower the exponents $p$ and $q$.
In our approach we adapt many ideas from \cite{MR2949622}, but we also use some fundamental martingale inequalities that only recently became available in the literature.
Consequently, we are even able to give a somewhat shorter proof.

\subsection{Continuous-time martingales}
One benefit of having Theorem~\ref{thm:main} formulated for general discrete-time martingales is that estimates \eqref{eq:2variation} and \eqref{eq:2jump} immediately transfer to continuous-time martingales $X=(X_t)_{t\geq 0}$ and $Y=(Y_t)_{t\geq 0}$. It is standard in stochastic calculus to assume that $X$ and $Y$ almost surely have \cadlag{} paths and that their filtration $\calF=(\calF_t)_{t\geq 0}$ satisfies ``the usual hypotheses'' from Protter's book \cite{MR2273672}, i.e.\ that $\calF_0$ is complete and that $\calF$ is right-continuous.
We fix the exponents $p,q,r$ satisfying \eqref{eq:exponents} and additionally assume that
\begin{equation}\label{eq:contLpassume}
\norm{X}_{\LL^p}:=\sup_{t\geq0}\norm{X_t}_{\LL^p}<\infty \quad\text{and}\quad \norm{Y}_{\LL^q}:=\sup_{t\geq0}\norm{Y_t}_{\LL^q}<\infty.
\end{equation}
Under more restrictive conditions on $X$ and $Y$, such as $\norm{X}_{\LL^\infty}<\infty$ and $\norm{Y}_{\LL^2}<\infty$, the papers \cite{MR955015} and \cite{MR3896720} proceed by defining the paraproduct as the process $\Pi(X,Y)=(\Pi_t(X,Y))_{t\geq 0}$ given by the stochastic integral
\begin{equation}\label{eq:integraltoapprox}
\Pi_t(X,Y) := \int_{(0,t]} X_{s-} \dif Y_s.
\end{equation}
Here $X_{s-}$ stands for the left limit $\lim_{u\to s-}X_u$. The above integral is understood as the It\^{o} integral and it yields another process with almost surely \cadlag{} paths.
In order to extend the definition of $\Pi(X,Y)$ to the martingales satisfying \eqref{eq:contLpassume} only, and to enable the application of Theorem~\ref{thm:main}, we prefer to construct the martingale paraproduct as a limit of certain discrete-time paraproducts, namely the Riemann sums of \eqref{eq:integraltoapprox}.

A \emph{random partition} of $[0,\infty)$ will be any tuple $\Sigma=(\tau_0,\tau_1,\ldots,\tau_l)$ of finite stopping times with respect to $\mathcal{F}$ such that $0=\tau_0\leq\tau_1\leq\cdots\leq\tau_l$.
We define the corresponding \emph{Riemann sum} of the integral \eqref{eq:integraltoapprox} as the stochastic process $S(X,Y;\Sigma)=(S_{t}(X,Y;\Sigma))_{t\geq 0}$ given by
\[ S_t(X,Y;\Sigma) := \sum_{j=1}^{l} X_{\min\{t,\tau_{j-1}\}} \big(Y_{\min\{t,\tau_j\}}-Y_{\min\{t,\tau_{j-1}\}}\big). \]
Following the language of \cite{MR2273672}, let us say that a sequence of random partitions $(\Sigma_n)_{n=1}^{\infty}$, $\Sigma_n=(\tau^{(n)}_0,\tau^{(n)}_1,\ldots,\tau^{(n)}_{l_n})$, \emph{tends to the identity} if $\lim_{n\to\infty}\tau^{(n)}_{l_n}=\infty$ a.s.\@ and $\lim_{n\to\infty}\max_{1\leq j\leq l_n}|\tau^{(n)}_j - \tau^{(n)}_{j-1}|=0$ a.s.

\begin{corollary}\label{cor:main}
\begin{enumerate}[(a)]
\item\label{it:cor:a} There exists a unique (up to indistinguishability) stochastic proces $\Pi(X,Y)=(\Pi_t(X,Y))_{t\geq 0}$ with a.s.\@ \cadlag{} paths such that for any sequence of random partitions $(\Sigma_n)_{n=1}^{\infty}$ tending to the identity the Riemann sums $S(X,Y;\Sigma_n)$ converge uniformly on compacts in probability (u.c.p.) towards $\Pi(X,Y)$, i.e.
\[ \lim_{n\to\infty} \mathbb{P}\Big(\sup_{s\in[0,t]}|S_s(X,Y;\Sigma_n)-\Pi_s(X,Y)|>\varepsilon\Big) = 0 \]
for each $\varepsilon>0$ and each $t>0$.
We say that $\Pi(X,Y)$ is the \emph{paraproduct} of martingales $X$ and $Y$.
\item\label{it:cor:b} \emph{Truncated paraproducts} are now defined as random variables
\[ \Pi_{t,t'}(X,Y) := \Pi_{t'}(X,Y) - \Pi_{t}(X,Y) - X_t (Y_{t'}-Y_t) \]
for any $0\leq t< t'<\infty$. We have
\begin{equation}\label{eq:2variationcont}
\norm[\bigg]{ \sup_{\substack{m\in\mathbb{N}\\ t_0<t_1<\cdots<t_m}} \Big( \sum_{k=1}^{m} \abs[\Big]{ \Pi_{t_{k-1},t_k} (X,Y) }^\varrho \Big)^{1/\varrho} }_{\LL^r}
\leq C_{p,q,\varrho} \norm{X}_{\LL^p} \norm{Y}_{\LL^q}
\end{equation}
for any $\varrho\in(1,\infty)$ and
\begin{multline}\label{eq:2jumpcont}
\norm[\Bigg]{ \sup_{\substack{m\in\mathbb{N}\\ t_0<t_1<\cdots<t_m}} \operatorname{card} \Set[\bigg]{ k\in\Set{1,\ldots,m} \given \abs[\Big]{ \Pi_{t_{k-1},t_k}(X,Y) } \geq \lambda } }_{\LL^r} \\
\leq C_{p,q} \lambda^{-1} \norm{X}_{\LL^p} \norm{Y}_{\LL^q}
\end{multline}
for any $\lambda\in(0,\infty)$.
\end{enumerate}
\end{corollary}

\subsection{Connection with rough paths}
One can view the triple
\[
H_{n} := (f_{n},g_{n},\Pi_{n}(f,g))
\]
as a process with values in a $3$-dimensional Heisenberg group $\mathbb{H} \cong \mathbb{R}^{3}$ with the group operation
\[ (x,y,z) \cdot (x',y',z') = (x+x',y+y',z+z'+xy'). \]
Then the truncated martingale paraproducts $\Pi_{n,n'}$ are precisely the $z$-coordinates of the differences of this process.
More precisely, for any times $n \leq n'$ we have
\[
H_{n} \cdot (f_{n'}-f_{n}, g_{n'}-g_{n}, \Pi_{n,n'}(f,g)) = H_{n'}.
\]
This corresponds to Chen's relation \cite[(2.1)]{MR3289027} in the theory of rough paths.

On the Heisenberg group we consider the homogeneous \emph{box norm} $\norm{(x,y,z)} := \max (\abs{x}, \abs{y}, \abs{z}^{1/2})$ and the corresponding distance function $d(H,H') := \norm{H^{-1} H'}$.
One can verify that $\norm{H \cdot H'} \leq \norm{H} + \norm{H'}$, and this implies the triangle inequality $d(H,H'') \leq d(H,H') + d(H',H'')$.
Any other left-invariant homogeneous distance, e.g.\ the Carnot--Carath\'eodory distance, would work equally well.
Combining \eqref{eq:2jump} and \eqref{eq:1jump} one can obtain the jump estimate
\begin{multline*}
\norm[\Bigg]{ \bigg( \sup_{\substack{m\in\mathbb{N}\\ n_0<n_1<\cdots<n_m}} \operatorname{card} \Set[\Big]{ k\in\Set{1,\ldots,m} \given d(H_{n_{k}},H_{n_{k-1}}) \geq \lambda } \bigg)^{1/2} }_{\LL^p}\\
\leq C_p \lambda^{-1} (\norm{f}_{\LL^p} + \norm{g}_{\LL^p}).
\end{multline*}
Either using this estimate and \cite[Lemma~2.17]{arxiv:1808.04592} or combining \eqref{eq:2variation} and \eqref{eq:1lepingle} one can also obtain the variational estimate
\begin{equation}\label{eq:rough-lepingle}
\norm[\Bigg]{ \sup_{\substack{m\in\mathbb{N}\\ n_0<n_1<\cdots<n_m}} \bigg( \sum_{k=1}^{m} d(H_{n_{k}},H_{n_{k-1}})^\varrho \bigg)^{1/\varrho} }_{\LL^p}
\leq C_{p,\varrho} (\norm{f}_{\LL^p} + \norm{g}_{\LL^p})
\end{equation}
for any $\varrho>2$.
The estimate \eqref{eq:rough-lepingle} for continuous martingales is a special case of a result of Friz and Victoir \cite[Theorem 14]{MR2483743}, and for general \cadlag{} martingales it is a special case of a result of Chevyrev and Friz \cite[Theorem 4.7]{MR3909973} with $F(x)=x^{r}$.
Indeed, in our notation these results can be stated as \eqref{eq:rough-lepingle} with $f,g$ replaced by $\textup{S}f,\textup{S}g$ on the right-hand side, where $\textup{S}$ denotes the martingale square function as in \eqref{eq:sq-fct}.
Hence the estimate \eqref{eq:2jumpcont} can be seen as an off-diagonal and endpoint version of the cited results.

\section{Some known martingale inequalities}
We begin this section with a few words on the notation. Then we review several known martingale inequalities that will be needed in subsequent sections. Some of them we could not find formulated elsewhere with exactly the same assumptions. However, the proofs of those inequalities are still quite straightforward using the results available in the existing literature and we include them for completeness.

For any two quantities $A,B\in[0,\infty]$ we will write $A\lesssim B$ when there exists an unimportant constant $C\in[0,\infty)$ such that $A\leq CB$. Furthermore, we will write $A\sim B$ if both $A\lesssim B$ and $B\lesssim A$ hold. Dependencies of the implicit constants on some parameters will be denoted in the subscripts of $\lesssim$ and $\sim$.
For real numbers $a$ and $b$ we will write
\[ a\wedge b:=\min\Set{a,b},\quad a\vee b:=\max\Set{a,b}. \]
We have already encountered the \emph{$\LL^p$-quasinorms} $h\mapsto\norm{h}_{\LL^p}$ in the introductory section, both for finite $p$ and for $p=\infty$. Recall that for a martingale $f=(f_n)_{n=0}^{\infty}$ the quantity $\norm{f}_{\LL^p}$ is defined by \eqref{eq:martLpnorm}.
Any nonnegative random variable $w$ gives rise to the \emph{weighted $\LL^p$-quasinorms}, given for $p\in(0,\infty)$ as
\[ \norm{h}_{\LL^p(w)} := \big( \EE{\abs{h}^p w} \big)^{1/p}. \]
On the other hand, the \emph{weak $\LL^p$-quasinorm} is defined as
\[ \norm{h}_{\LL^{p,\infty}} := \Big( \sup_{\alpha\in(0,\infty)} \alpha^p \,\mathbb{P}(\abs{h}>\alpha) \Big)^{1/p} \]
for any $p\in(0,\infty)$.
Any sequence of random variables $h=(h^{(k)})_{k=1}^{\infty}$ can be regarded as a vector-valued random element and for $p\in(0,\infty]$ and $q\in(0,\infty)$ we define the \emph{mixed $\LL^p(\ell^q)$-quasinorm}
\[ \norm{h}_{\LL^p(\ell^q)} = \norm[\big]{h^{(k)}}_{\LL^p(\ell_k^q)} := \norm[\bigg]{ \Big( \sum_{k=1}^{\infty} \abs[\big]{h^{(k)}}^q \Big)^{1/q} }_{\LL^p}. \]
Finally, $p'$ will always denote the conjugated exponent of $p\in[1,\infty]$, i.e.\@ the unique number $p'\in[1,\infty]$ such that $1/p+1/p'=1$.

For any martingale $f=(f_n)_{n=0}^{\infty}$ with respect to a fixed filtration $\calF=(\calF_n)_{n=0}^{\infty}$ one defines the \emph{maximal function}
\[ \textup{M}f := \sup_{n\geq 0} \abs{f_n} \]
and the \emph{square function}
\begin{equation}\label{eq:sq-fct}
\textup{S}f := \Big( \sum_{n=1}^{\infty} \abs{\dif f_n}^2 \Big)^{1/2}.
\end{equation}
Note that $\textup{M}f$ and $\textup{S}f$ are two random variables taking values in $[0,\infty]$. In different terminology these are the limits of the maximum process of $f$ and the quadratic variation of $f$, respectively. If we start merely from a random variable $h$, then we automatically assign to it the martingale $(h_n)_{n=0}^{\infty}$ defined by $h_n:=\EE{ h \given \calF_n }$, so $\textup{M}h$ and $\textup{S}h$ still make sense.

The well known \emph{Burkholder--Davis--Gundy inequality} claims that
\begin{equation}\label{eq:BDG}
\norm{\textup{M}f}_{\LL^p} \sim_p \norm{\textup{S}f}_{\LL^p}
\end{equation}
for every $p\in[1,\infty)$. Indeed, the case $p>1$ is due to Burkholder \cite{MR0208647}, while the case $p=1$ was shown by Davis \cite{MR0268966}.
A weighted version of the latter case was established by Os\k{e}kowski \cite{MR3688518}:
\begin{equation}\label{eq:BDGweighted}
\norm{\textup{M}f}_{\LL^1(w)} \lesssim \norm{\textup{S}f}_{\LL^1(\textup{M}w)},
\end{equation}
where $w$ is a nonnegative integrable random variable, interpreted as a weight. The implicit constant in \eqref{eq:BDGweighted} is an absolute one and Os\k{e}kowski could choose \mbox{$16(\sqrt{2}+1)$}. Inequality \eqref{eq:BDGweighted} can also be viewed as a probabilistic analogue of a weighted estimate by Fefferman and Stein \cite{MR0284802}.

Moreover, \emph{Doob's maximal inequality} reads
\begin{equation}\label{eq:Doob}
\norm{\textup{M}f}_{\LL^p} \leq p' \norm{f}_{\LL^p}
\end{equation}
for every $p\in(1,\infty]$. It also has a weighted version, formulated for instance as a part of Theorem~3.2.3 in the book \cite{MR3617205}:
\begin{equation}\label{eq:Doobweighted}
\norm{\textup{M}f}_{\LL^p(w)} \leq p' \norm{f_{\infty}}_{\LL^p(\textup{M}w)}
\end{equation}
for $p\in(1,\infty]$. In \eqref{eq:Doobweighted} we assume, for convenience, that $(f_n)_{n=0}^{\infty}$ eventually becomes a constant sequence, so that $f_{\infty} := \lim_{n\to\infty} f_n$ trivially makes sense with respect to every possible mode of convergence.

Suppose that $T_0\leq T_1\leq T_2\leq\cdots$ is a sequence of stopping times taking values in $\mathbb{N}_0$ with respect to the filtration $\calF$ and assume that each $T_k$ is bounded. These stopping times will be used for the purpose of certain ``localization.'' Boundedness of each individual $T_k$ is a convenient assumption for the application of the optional sampling theorem; see e.g.~\cite{MR2059709}.
For every $k\in\mathbb{N}$ and every $n\in\mathbb{N}_0$ we note that $(n\vee T_{k-1})\wedge T_k$ is also a stopping time with respect to $\calF$ and define
\begin{equation}\label{eq:newfiltrations}
\calF^{(k)}_n := \calF_{(n\vee T_{k-1})\wedge T_k}.
\end{equation}
That way, each $\calF^{(k)}:=(\calF^{(k)}_n)_{n=0}^{\infty}$ becomes a filtration of the original probability space and each of these sequences of $\sigma$-algebras becomes constant for sufficiently large indices $n$.

\begin{lemma}\label{lm:lqdoobaux}
Let $(T_k)_{k=0}^{\infty}$ be an increasing sequence of bounded stopping times, let $(\calF^{(k)})_{k=1}^{\infty}$ be a sequence of filtrations defined by \eqref{eq:newfiltrations}, and for each $k\in\mathbb{N}$ let $f^{(k)}=(f^{(k)}_n)_{n=0}^{\infty}$ be a martingale with respect to $\calF^{(k)}$ that eventually becomes a constant sequence. For any $p,q\in(1,\infty)$ we have
\begin{equation}\label{eq:lqdoobaux}
\norm[\big]{ \textup{M}f^{(k)} }_{\LL^p(\ell^q_k)} \lesssim_{p,q} \norm[\big]{ f^{(k)}_{\infty} }_{\LL^p(\ell^q_k)}.
\end{equation}
\end{lemma}

Lemma~\ref{lm:lqdoobaux} can be viewed as an $\ell^q$-valued extension of Doob's maximal inequality \eqref{eq:Doob}. The proof of \eqref{eq:lqdoobaux} is based on \eqref{eq:Doobweighted} and it already exists as the proof of \cite[Theorem~3.2.7]{MR3617205}.
However, the working assumption in \cite{MR3617205} is that $f^{(k)}$ are arbitrary martingales with respect to the same filtration, which is not the case here.
For this reason and for the sake of completeness we prefer to repeat the short argument, rather than just invoke the result from \cite{MR3617205}.

\begin{proof}[Proof of Lemma~\ref{lm:lqdoobaux}.]
The case $p\geq q$ is handled first. Let $r\in(1,\infty]$ denote the conjugated exponent of $p/q$.
To an arbitrary random variable $w\geq 0$ satisfying $\norm{w}_{\LL^{r}} = 1$ we associate the martingales $(w_n)_{n=0}^{\infty}$ and $w^{(k)}=(w^{(k)}_n)_{n=0}^{\infty}$, for each $k\in\mathbb{N}$, via
\begin{equation}\label{eq:maxfnsofw}
w_n := \EE{w \given \calF_n}, \quad w^{(k)}_n := \EE{w \given \calF^{(k)}_n}.
\end{equation}
Consider the expression
\[ \EE[\bigg]{ \Big( \sum_{k=1}^{\infty} \big(\textup{M}f^{(k)}\big)^q \Big) w }
= \sum_{k=1}^{\infty} \EE[\bigg]{ \big(\textup{M}f^{(k)}\big)^q w }
= \sum_{k=1}^{\infty} \norm[\big]{ \textup{M}f^{(k)} }_{\LL^{q}(w)}^q. \]
By \eqref{eq:Doobweighted} this is at most a constant depending on $q$ times
\begin{align*}
& \sum_{k=1}^{\infty} \norm[\big]{ f^{(k)}_{\infty} }_{\LL^{q}(\textup{M}w^{(k)})}^q
= \sum_{k=1}^{\infty} \EE[\Big]{ \abs[\big]{f^{(k)}_{\infty}}^q \big(\textup{M}w^{(k)}\big) } \\
& \leq \EE[\Big]{ \Big( \sum_{k=1}^{\infty} \abs[\big]{f^{(k)}_{\infty}}^q \Big) \big(\textup{M} w\big) }
\leq \norm[\Big]{ \sum_{k=1}^{\infty} \abs[\big]{f^{(k)}_{\infty}}^q }_{\LL^{p/q}} \norm[\big]{\textup{M}w}_{\LL^{r}}.
\end{align*}
Applying \eqref{eq:Doob} to $\norm{\textup{M}w}_{\LL^{r}}$ and recalling the freedom that we had in choosing $w$, we obtain
\[ \norm[\Big]{ \sum_{k=1}^{\infty} \big(\textup{M}f^{(k)}\big)^q }_{\LL^{p/q}} \lesssim_{p,q} \norm[\Big]{ \sum_{k=1}^{\infty} \abs[\big]{f^{(k)}_{\infty}}^q }_{\LL^{p/q}}, \]
which transforms into \eqref{eq:lqdoobaux} after taking the $q$-th root of both sides.

Turning to the case $p\leq q$, we take some $r\in(1,p)$ and denote $a:=p/r\in(1,\infty)$, $b:=q/r\in(1,\infty)$. Write
\begin{equation}\label{eq:lqdoobcalc}
\norm[\big]{ \textup{M}f^{(k)} }_{\LL^p(\ell^q_k)}^r
= \bigg( \EE{ \sum_{k=1}^{\infty} \big(\textup{M}f^{(k)}\big)^q }^{p/q} \bigg)^{r/p}
= \norm[\big]{ \big(\textup{M}f^{(k)}\big)^{r} }_{\LL^a(\ell^b_k)}.
\end{equation}
We are going to dualize the mixed $\LL^a(\ell^b)$-norm above and for this we take a sequence $h=(h^{(k)})_{k=1}^{\infty}$ of nonnegative random variables such that $\norm{h}_{\LL^{a'}(\ell^{b'})} <\infty$. Each $h^{(k)}$ defines a martingale $(h^{(k)}_n)_{n=0}^{\infty}$ by $h^{(k)}_n:=\EE{h^{(k)} \given \calF^{(k)}_n}$.
Using \eqref{eq:Doobweighted} for each fixed $k$ followed by H\"{o}lder's inequality we obtain
\begin{align*}
\mathbb{E} \sum_{k=1}^{\infty} \big(\textup{M}f^{(k)}\big)^r h^{(k)}
& = \sum_{k=1}^{\infty} \norm[\big]{ \textup{M}f^{(k)} }_{\LL^{r}(h^{(k)})}^r
\lesssim_r \sum_{k=1}^{\infty} \norm[\big]{ f^{(k)}_{\infty} }_{\LL^{r}(\textup{M}h^{(k)})}^r \\
&= \mathbb{E} \sum_{k=1}^{\infty} \abs[\big]{f^{(k)}_{\infty}}^r \big(\textup{M}h^{(k)}\big)
\leq \norm[\big]{ \abs[\big]{f^{(k)}_{\infty}}^r }_{\LL^{a}(\ell^{b}_k)} \norm[\big]{\textup{M}h^{(k)}}_{\LL^{a'}(\ell^{b'}_k)}.
\end{align*}
Then applying the previous case of \eqref{eq:lqdoobaux} (with $p,q$ replaced by $a',b'$) to get
\[ \norm[\big]{\textup{M}h^{(k)}}_{\LL^{a'}(\ell^{b'}_k)} \lesssim_{a,b} \norm[\big]{ h^{(k)}_{\infty} }_{\LL^{a'}(\ell^{b'}_k)} = \norm{h}_{\LL^{a'}(\ell^{b'})} \]
and using duality we end up with
\[ \norm[\big]{ \big(\textup{M}f^{(k)}\big)^{r} }_{\LL^a(\ell^b_k)}
\lesssim_{a,b} \norm[\big]{ \abs[\big]{f^{(k)}_{\infty}}^r }_{\LL^{a}(\ell^{b}_k)}. \]
Recall the computation \eqref{eq:lqdoobcalc} and take the $r$-th root of both sides.
\end{proof}

Now, let $f=(f_n)_{n=0}^{\infty}$ be a single martingale with respect to $\calF$. For every $k\in\mathbb{N}$ and every $n\in\mathbb{N}_0$ we denote, for the rest of the paper,
\begin{equation}\label{eq:newmartingales}
f^{(k)}_n := f_{(n\vee T_{k-1})\wedge T_k} - f_{T_{k-1}},
\end{equation}
i.e.
\[ f^{(k)}_n
= \begin{cases} 0 & \text{for } n\leq T_{k-1}, \\
f_n - f_{T_{k-1}} & \text{for } T_{k-1}<n\leq T_k, \\
f_{T_k} - f_{T_{k-1}} & \text{for } n>T_k. \end{cases} \]
That way, for each $k\in\mathbb{N}$ we have now defined a particular martingale $f^{(k)}:=(f^{(k)}_n)_{n=0}^{\infty}$ with respect to the filtration $\calF^{(k)}$ given by \eqref{eq:newfiltrations}. It is ``interesting'' only for moments between $T_{k-1}$ and $T_k$. Consequently, the sequence $(f^{(k)}_n)_{n=0}^{\infty}$ eventually becomes constant and, in particular, the limit $f^{(k)}_\infty := \lim_{n\to\infty} f^{(k)}_n$ exists (in every possible way) and simply equals $f_{T_k} - f_{T_{k-1}}$.
Many classical inequalities in terms of martingale $f$ have their vector-valued extensions in terms of its ``localized pieces'' $f^{(k)}$. Our next goal is to formulate and prove a couple of those, as they will be needed in the next section.

\begin{lemma}\label{lm:auxineq}
Let $(T_k)_{k=0}^{\infty}$ be an increasing sequence of bounded stopping times and let $f$ be a martingale, both with respect to $\calF$. Moreover, let $(f^{(k)})_{k=1}^{\infty}$ be a sequence of martingales defined by \eqref{eq:newmartingales}.
\begin{enumerate}[(a)]
\item\label{it:auxin:a} For any $p\in(1,\infty)$ we have
\begin{equation}\label{eq:l2Doob}
\norm[\big]{ \textup{M}f^{(k)} }_{\LL^p(\ell^2_k)} \lesssim_p \norm{f}_{\LL^p}.
\end{equation}
\item\label{it:auxin:b} For any $p\in(1,\infty)$ we have
\begin{equation}\label{eq:l2Square}
\norm[\big]{ \textup{S}f^{(k)} }_{\LL^p(\ell^2_k)} \lesssim_p \norm{f}_{\LL^p}.
\end{equation}
\end{enumerate}
\end{lemma}

\begin{proof}[Proof of Lemma~\ref{lm:auxineq}.]
\eqref{it:auxin:a} Since the stopping times $T_k$ are bounded, using the optional sampling theorem (see Section 12.4 of the book \cite{MR2059709}) and applying \eqref{eq:BDG} and \eqref{eq:Doob} to the ``optionally sampled'' martingale $(f_{T_n})_{n=0}^{\infty}$ we get
\begin{align*}
\norm[\big]{ f^{(k)}_{\infty} }_{\LL^p(\ell^2_k)} & = \norm[\Big]{ \Big( \sum_{k=1}^{\infty} \abs[\big]{ f^{(k)}_{\infty} }^2 \Big)^{1/2} }_{\LL^p}
= \norm[\Big]{ \Big( \sum_{k=1}^{\infty} \abs[\big]{ f_{T_k} - f_{T_{k-1}} }^2 \Big)^{1/2} }_{\LL^p} \\
& = \norm{\textup{S}(f_{T_k})_{k=0}^{\infty}}_{\LL^p} \lesssim_p \norm{\textup{M}(f_{T_k})_{k=0}^{\infty}}_{\LL^p} \lesssim_p \norm{f}_{\LL^p}.
\end{align*}
Combining this with estimate \eqref{eq:lqdoobaux} from Lemma~\ref{lm:lqdoobaux} specialized to $q=2$ establishes \eqref{eq:l2Doob}.

\eqref{it:auxin:b} Estimate \eqref{eq:l2Square} is immediate. We only need to observe
\[ \textup{S}f^{(k)} = \Big( \sum_{T_{k-1}<n\leq T_{k}} \abs{\dif f_n}^2 \Big)^{1/2} \]
and
\[ \Big( \sum_{k=1}^{\infty} \big( \textup{S}f^{(k)} \big)^2 \Big)^{1/2} \leq \textup{S}f, \]
and then apply \eqref{eq:BDG} and \eqref{eq:Doob}:
\[ \norm[\big]{ \textup{S}f^{(k)} }_{\LL^p(\ell^2_k)} \leq \norm{\textup{S}f}_{\LL^p} \lesssim_p \norm{\textup{M}f}_{\LL^p} \lesssim_p \norm{f}_{\LL^p}. \qedhere \]
\end{proof}

\subsection{Multilinear interpolation}
We will repeatedly use a multilinear version of the Marcinkiewicz interpolation theorem.
We caution the reader that many such results exist in the literature, and not every version would be adequate for our purposes.
We refer to \cite[Corollary 1.1]{MR2876406}, of which the result below is a special case, although it also follows e.g.\ from the result of \cite{MR942274} on abstract interpolation spaces.
\begin{theorem}\label{thm:interpolation}
Let $T$ be a bisublinear operator, i.e., $\abs{T(f_{1}+f_{2},g)} \leq \abs{T(f_{1},g)} + \abs{T(f_{2},g)}$ and $\abs{T(f,g_{1}+g_{2})} \leq \abs{T(f,g_{1})} + \abs{T(f,g_{2})}$, initially defined on simple functions on a pair of measure spaces with values in measurable functions on a third measure space.
Suppose that the estimate
\begin{equation}\label{eq:interpolation:hypothesis}
\norm{T(f,g)}_{\LL^{r,\infty}}
\leq C \norm{f}_{\LL^{p}} \norm{g}_{\LL^{q}}
\end{equation}
holds with $0 < p,q \leq \infty$, $1/r = 1/p+1/q$, and $(1/p,1/q)$ being the corners of a non-degenerate triangle $\Delta \subset [0,\infty)^{2}$.
Then for every $0 < p,q \leq \infty$ such that $(1/p,1/q)$ is in the interior of $\Delta$ and for $1/r = 1/p+1/q$ we have
\begin{equation*}
\norm{T(f,g)}_{\LL^{r}}
\leq C \norm{f}_{\LL^{p}} \norm{g}_{\LL^{q}},
\end{equation*}
where the constant $C$ depends only on $\Delta,p,q$, and the constants in \eqref{eq:interpolation:hypothesis}.
\end{theorem}

\section{A vector-valued estimate for martingale paraproducts}
The main ingredient in the proof of Theorem~\ref{thm:main} is the following proposition.

\begin{proposition}\label{prop:l1pprod}
Let $(T_k)_{k=0}^{\infty}$ be an increasing sequence of bounded stopping times and let $f$ and $g$ be martingales, all with respect to the same filtration $\calF$. Moreover, let $(f^{(k)})_{k=1}^{\infty}$ and $(g^{(k)})_{k=1}^{\infty}$ be sequences of martingales defined from $f$ and $g$, respectively, via \eqref{eq:newmartingales}.
Then for any exponents $p,q,r$ satisfying \eqref{eq:exponents} we have the estimate
\begin{equation}\label{eq:l1pprod}
\norm[\big]{ \Pi_{\infty}(f^{(k)},g^{(k)}) }_{\LL^r(\ell^1_k)} \lesssim_{p,q} \norm{f}_{\LL^p} \norm{g}_{\LL^q}.
\end{equation}
\end{proposition}

Note that, for each $k\in\mathbb{N}$, the paraproduct $\Pi(f^{(k)},g^{(k)})$ is a martingale with respect to the filtration $\calF^{(k)}$ given by \eqref{eq:newfiltrations}. The sequence $(\Pi_n(f^{(k)},g^{(k)}))_{n=0}^{\infty}$ eventually becomes constant, so that $\Pi_{\infty}(f^{(k)},g^{(k)})$ makes sense. A crucial observation, following from \eqref{eq:pprodaltexp} and needed later, is
\begin{equation}\label{eq:pprodpieces}
\Pi_{\infty}(f^{(k)},g^{(k)}) = \sum_{T_{k-1}<i<j\leq T_{k}} \dif f_i \dif g_j
\end{equation}
and these are precisely the truncated paraproducts appearing on the left hand sides of estimates \eqref{eq:2variation} and \eqref{eq:2jump} if we replace each $n_k$ with $T_k$.

\begin{proof}[Proof of Proposition~\ref{prop:l1pprod}.]
Let us first discuss the case $r\geq 1$ of estimate \eqref{eq:l1pprod}. We begin by proving the $\ell^1$-valued estimate
\begin{equation}\label{eq:l1BDG}
\norm[\big]{ \Pi_{\infty}(f^{(k)},g^{(k)}) }_{\LL^r(\ell^1_k)} \lesssim_p \norm[\big]{ \textup{S}\Pi(f^{(k)},g^{(k)}) }_{\LL^r(\ell^1_k)}.
\end{equation}
We will reduce it to the weighted estimate \eqref{eq:BDGweighted} for martingales $\Pi(f^{(k)},g^{(k)})$.
Take an arbitrary nonnegative random variable satisfying $\norm{w}_{\LL^{r'}} = 1$ and define $(w_n)_{n=0}^{\infty}$ and $w^{(k)}=(w^{(k)}_n)_{n=0}^{\infty}$ as in \eqref{eq:maxfnsofw}. We have
\begin{align*}
& \EE[\bigg]{ \Big( \sum_{k=1}^{\infty} \abs[\big]{\Pi_{\infty}(f^{(k)},g^{(k)})} \Big) w }
\leq \EE[\bigg]{ \Big( \sum_{k=1}^{\infty} \textup{M}\Pi(f^{(k)},g^{(k)}) \Big) w } \\
& = \sum_{k=1}^{\infty} \EE[\bigg]{ \big(\textup{M}\Pi(f^{(k)},g^{(k)})\big) w }
= \sum_{k=1}^{\infty} \norm[\big]{ \textup{M}\Pi(f^{(k)},g^{(k)}) }_{\LL^{1}(w)}
\end{align*}
and, by \eqref{eq:BDGweighted} applied to martingale $\Pi(f^{(k)},g^{(k)})$ for each fixed $k$, this is at most a constant times
\begin{align*}
& \sum_{k=1}^{\infty} \norm[\big]{ \textup{S}\Pi(f^{(k)},g^{(k)}) }_{\LL^{1}(\textup{M}w^{(k)})}
= \sum_{k=1}^{\infty} \EE[\bigg]{ \big(\textup{S}\Pi(f^{(k)},g^{(k)})\big) \big(\textup{M}w^{(k)}\big) } \\
& \leq \EE[\bigg]{ \Big(\sum_{k=1}^{\infty}\textup{S}\Pi(f^{(k)},g^{(k)})\Big) \big(\textup{M} w\big) }
\leq \norm[\big]{ \textup{S}\Pi(f^{(k)},g^{(k)}) }_{\LL^r(\ell^1_k)} \norm[\big]{\textup{M}w}_{\LL^{r'}}.
\end{align*}
Using Doob's inequality \eqref{eq:Doob} for the martingale $(w_n)_{n=0}^{\infty}$ we end up with
\[ \EE[\bigg]{ \Big( \sum_{k=1}^{\infty} \abs[\big]{\Pi_{\infty}(f^{(k)},g^{(k)})} \Big) w }
\lesssim r
\norm[\big]{ \textup{S}\Pi(f^{(k)},g^{(k)}) }_{\LL^r(\ell^1_k)}. \]
Recalling the freedom that we had in choosing $w$ we establish \eqref{eq:l1BDG} by dualization.

In order to complete the proof of \eqref{eq:l1pprod} in the case $r\geq 1$, observe that the expression on the right hand side of \eqref{eq:l1BDG} is, by the definition of the paraproduct, equal to
\[ \norm[\Big]{ \sum_{k=1}^{\infty} \Big( \sum_{n=1}^{\infty} (f^{(k)}_{n-1})^2 (\dif g^{(k)}_{n})^2 \Big)^{1/2} }_{\LL^r}
\leq \norm[\Big]{ \sum_{k=1}^{\infty} \textup{M}f^{(k)} \textup{S}g^{(k)} }_{\LL^r}, \]
which is, by H\"{o}lder's inequality, in turn bounded by
\[ \norm[\big]{ \textup{M}f^{(k)} }_{\LL^p(\ell^2_k)} \norm[\big]{ \textup{S}g^{(k)} }_{\LL^q(\ell^2_k)}
\lesssim_{p,q} \norm{f}_{\LL^p} \norm{g}_{\LL^q}. \]
In the last inequality we used \eqref{eq:l2Doob} and \eqref{eq:l2Square} for the martingales $f$ and $g$, respectively.

We will now prove the weak-type estimate
\begin{equation}\label{eq:weakest}
\norm[\big]{ \Pi_{\infty}(f^{(k)},g^{(k)}) }_{\LL^{r,\infty}(\ell^1_k)} \lesssim_{p} \norm{f}_{\LL^p} \norm{g}_{\LL^1}
\end{equation}
for any $p\in(1,\infty)$ and $r\in(1/2,1)$ such that $1/p+1=1/r$.
This will conclude the proof of \eqref{eq:l1pprod} for $r<1$ by real interpolation with the previously established cases (Theorem~\ref{thm:interpolation}).
By the homogeneity of \eqref{eq:weakest} we can normalize: assume $\norm{f}_{\LL^p}=1$ and $\norm{g}_{\LL^1}=1$.
Fix a number $\nu>0$ and perform Gundy's decomposition \cite{MR0221573} of the martingale $g$ at height $\alpha=\nu^{r}$; see its formulation as Theorem~3.4.1 in the book \cite{MR3617205}.
It splits $g$ as
\[ g_n = g^{\textup{good}}_n + g^{\textup{bad}}_n + g^{\textup{harmless}}_n, \]
where $g^{\textup{good}}=(g^{\textup{good}}_n)_{n=0}^{\infty}$, $g^{\textup{bad}}=(g^{\textup{bad}}_n)_{n=0}^{\infty}$, and $g^{\textup{harmless}}=(g^{\textup{harmless}}_n)_{n=0}^{\infty}$ are martingales with respect to $\calF$ satisfying
\begin{align}
& g^{\textup{good}}_0 = g_0, \quad g^{\textup{bad}}_0 = g^{\textup{harmless}}_0 = 0, \nonumber \\
& \norm{g^{\textup{good}}}_{\LL^\infty}\leq 2\alpha, \quad \norm{g^{\textup{good}}}_{\LL^1}\leq 4\norm{g}_{\LL^1}, \label{eq:gundy1} \\
& \mathbb{P}(\textup{M}g^{\textup{bad}}>0) \leq 3 \alpha^{-1} \norm{g}_{\LL^1}, \label{eq:gundy2} \\
& \sum_{n=1}^{\infty} \norm{\dif g^{\textup{harmless}}_n}_{\LL^1} \leq 4 \norm{g}_{\LL^1}. \label{eq:gundy3}
\end{align}
Construct the martingales $g^{\textup{good},(k)}$, $g^{\textup{bad},(k)}$, and $g^{\textup{harmless},(k)}$ for the given sequence of stopping times via the formula \eqref{eq:newmartingales}.
Using the previously established case $r=1$ of estimate \eqref{eq:l1pprod} and \eqref{eq:gundy1} we obtain
\begin{align*}
& \mathbb{P}\Big(\sum_{k=1}^{\infty} \abs[\big]{\Pi_{\infty}(f^{(k)},g^{\textup{good},(k)})} > \frac{\nu}{2} \Big)
\lesssim \nu^{-1} \norm[\Big]{ \sum_{k=1}^{\infty} \abs[\big]{\Pi_{\infty}(f^{(k)},g^{\textup{good},(k)})} }_{\LL^1} \\
& \lesssim_p \nu^{-1} \norm{f}_{\LL^p} \norm{g^{\textup{good}}}_{\LL^{p'}}
\leq \nu^{-1} \norm{f}_{\LL^p} \norm{g^{\textup{good}}}_{\LL^{\infty}}^{1/p} \norm{g^{\textup{good}}}_{\LL^{1}}^{1/p'}
\lesssim \nu^{-1} \nu^{r/p} = \nu^{-r}.
\end{align*}
Next, \eqref{eq:gundy2} yields
\[ \mathbb{P}\Big(\sum_{k=1}^{\infty} \abs[\big]{\Pi_{\infty}(f^{(k)},g^{\textup{bad},(k)})} > 0 \Big)
\leq \mathbb{P}(\textup{M}g^{\textup{bad}}>0) \lesssim \nu^{-r}. \]
Finally, by H\"{o}lder's inequality, Doob's inequality \eqref{eq:Doob}, and \eqref{eq:gundy3} we conclude
\begin{align*}
& \mathbb{P}\Big(\sum_{k=1}^{\infty} \abs[\big]{\Pi_{\infty}(f^{(k)},g^{\textup{harmless},(k)})} > \frac{\nu}{2} \Big)
\lesssim_r \nu^{-r} \norm[\Big]{ \sum_{k=1}^{\infty} \abs[\big]{\Pi_{\infty}(f^{(k)},g^{\textup{harmless},(k)})} }_{\LL^r}^r \\
& \leq \nu^{-r} \norm[\Big]{ \sum_{k=1}^{\infty} \sum_{T_{k-1}<j\leq T_k} \abs{f_{j-1}-f_{T_{k-1}}} \abs{\dif g^{\textup{harmless}}_j} }_{\LL^r}^r \\
& \lesssim \nu^{-r} \norm[\Big]{ (\textup{M}f) \sum_{n=1}^{\infty} \abs{\dif g^{\textup{harmless}}_n} }_{\LL^r}^r
\leq \nu^{-r} \norm{\textup{M}f}_{\LL^p}^r \norm[\Big]{ \sum_{n=1}^{\infty} \abs{\dif g^{\textup{harmless}}_n} }_{\LL^1}^r
\lesssim_{p,r} \nu^{-r}.
\end{align*}
Combining the above three estimates finishes the proof of \eqref{eq:weakest}.
\end{proof}

\section{Proof of variational and jump inequalities}

\begin{proof}[Proof of Theorem~\ref{thm:main}.]
In the process of proving estimates \eqref{eq:2variation} and \eqref{eq:2jump} we can constrain the numbers $n_0,n_1,\ldots,n_m$ to a finite interval of integers $\Set{0,1,2,\ldots,n_{\textup{max}}}$. Then we only need to take care that the obtained constants do not depend on $n_{\textup{max}}$. Afterwards we will be able to let $n_{\textup{max}}\to\infty$ and use the monotone convergence theorem, recovering Theorem~\ref{thm:main} in its full generality.

Let us begin with a stopping time argument enabling us to apply Proposition~\ref{prop:l1pprod}. We are given two martingales, $f=(f_n)_{n=0}^{\infty}$ and $g=(g_n)_{n=0}^{\infty}$, with respect to the filtration $\calF$. Fix $\lambda>0$ and recursively define an increasing sequence of stopping times $(S_k)_{k=0}^{\infty}$ by setting $S_0:=0$ and
\begin{multline*}
S_k := \min\bigg\{ n>S_{k-1} \bigg| \abs[\Big]{\sum_{S_{k-1}<i<j\leq n} \dif f_i \dif g_j} \geq \frac{\lambda}{3} \text{ or } \\
\max_{n'\in(S_{k-1},n]}\abs[\big]{f_{n'}-f_{S_{k-1}}}\abs[\big]{g_n-g_{n'}} \geq\frac{\lambda}{3} \bigg\},
\end{multline*}
with the convention $\min\emptyset=\infty$. Then for each $k\in\mathbb{N}_0$ set $T_k := S_k \wedge n_{\textup{max}}$.

Denote by $N_{\lambda}(f,g)$ the supremum of cardinalities on the left hand side of \eqref{eq:2jump}, so that the desired estimate \eqref{eq:2jump} becomes
\[ \norm{ N_{\lambda}(f,g) }_{\LL^r} \lesssim_{p,q} \lambda^{-1} \norm{f}_{\LL^p} \norm{g}_{\LL^q}. \]
On the other hand, denote
\[ \widetilde{N}_{\lambda}(f,g) := \sup \Set{ k\in\mathbb{N}_0 \given S_k \leq n_{\textup{max}} }. \]
Let us show that
\begin{equation}\label{eq:comparingjumps}
N_{\lambda}(f,g) \leq \widetilde{N}_{\lambda}(f,g)
\end{equation}
and for this it is sufficient to show that each interval of integers $(n',n'']\subseteq(0,n_{\textup{max}}]$ such that
\begin{equation}\label{eq:lambdajump}
\abs[\Big]{ \sum_{n'<i<j\leq n''} \dif f_i \dif g_j } \geq \lambda
\end{equation}
has to contain at least one of the stopping times $(S_k)_{k=1}^{\infty}$. If that was not the case, then we could choose an index $k\in\mathbb{N}$ such that $S_{k-1} \leq n' < n'' < S_k$, where we allow $S_k$ to be infinite.
Let us use the identity
\begin{align*}
\sum_{n'<i<j\leq n''} \dif f_i \dif g_j =\
& \sum_{S_{k-1}<i<j\leq n''} \dif f_i \dif g_j - \sum_{S_{k-1}<i<j\leq n'} \dif f_i \dif g_j \\
& - \big(f_{n'}-f_{S_{k-1}}\big)\big(g_{n''}-g_{n'}\big)
\end{align*}
and the fact that $S_k$ is strictly larger than $n'$ and $n''$, which implies that each of the three terms on the right hand side is strictly less than $\lambda/3$ in the absolute value.
That way we arrive at a contradiction with \eqref{eq:lambdajump} and complete the proof of \eqref{eq:comparingjumps}.

We plan to apply Proposition~\ref{prop:l1pprod} with the above sequence of stopping times $(T_k)_{k=0}^{\infty}$.
By the definitions of $S_k$ and $T_k$ we have
\begin{align*}
\lambda \widetilde{N}_{\lambda}(f,g) \leq\
& 3 \sum_{k=1}^{\widetilde{N}_{\lambda}(f,g)} \abs[\Big]{\sum_{T_{k-1}<i<j\leq T_{k}} \dif f_i \dif g_j} \\
& + 3 \sum_{k=1}^{\widetilde{N}_{\lambda}(f,g)} \max_{n'\in(T_{k-1},T_{k}]}\abs[\big]{f_{n'}-f_{T_{k-1}}}\abs[\big]{g_{T_{k}}-g_{n'}}.
\end{align*}
In the first term above we use \eqref{eq:pprodpieces}, while the second term is bounded by
\begin{align*}
& 6 \sum_{k=1}^{\widetilde{N}_{\lambda}(f,g)} \Big(\max_{n\in[T_{k-1},T_{k}]}\abs[\big]{f_{n}-f_{T_{k-1}}}\Big) \Big(\max_{n\in[T_{k-1},T_{k}]}\abs[\big]{g_{n}-g_{T_{k-1}}}\Big) \\
& \lesssim \sum_{k=1}^{\infty} \big(\textup{M} f^{(k)}\big) \big(\textup{M} g^{(k)}\big).
\end{align*}
Altogether, by the Cauchy--Schwarz inequality,
\[ \lambda \widetilde{N}_{\lambda}(f,g) \lesssim \sum_{k=1}^{\infty} \abs[\big]{ \Pi_{\infty}(f^{(k)},g^{(k)}) }
+ \Big(\sum_{k=1}^{\infty} \big(\textup{M} f^{(k)}\big)^2 \Big)^{1/2} \Big(\sum_{k=1}^{\infty} \big(\textup{M} g^{(k)}\big)^2 \Big)^{1/2}, \]
so using H\"older's inequality, \eqref{eq:l1pprod}, and \eqref{eq:l2Doob} we obtain
\begin{align*}
\norm{ \lambda \widetilde{N}_{\lambda}(f,g) }_{\LL^r}
& \lesssim \norm[\big]{ \Pi_{\infty}(f^{(k)},g^{(k)}) }_{\LL^r(\ell^1_k)} + \norm[\big]{ \textup{M}f^{(k)} }_{\LL^p(\ell^2_k)} \norm[\big]{ \textup{M}g^{(k)} }_{\LL^q(\ell^2_k)} \\
& \lesssim_{p,q} \norm{f}_{\LL^p} \norm{g}_{\LL^q}.
\end{align*}
Recalling \eqref{eq:comparingjumps} we complete the proof of the jump inequality \eqref{eq:2jump}.

By \cite[Lemma~2.17]{arxiv:1808.04592} the jump estimate \eqref{eq:2jump} immediately implies the weak type $\LL^{p}\times\LL^{q}\to\LL^{r,\infty}$ analogue of \eqref{eq:2variation}.
The strong type $\varrho$-variational estimate \eqref{eq:2variation} then follows by real interpolation for multisublinear operators (Theorem~\ref{thm:interpolation}).
\end{proof}

\section{Continuous-time martingales}
\begin{proof}[Proof of Corollary~\ref{cor:main}.]
\eqref{it:cor:a} In the particular case $\norm{X}_{\LL^\infty}<\infty$ and $\norm{Y}_{\LL^2}<\infty$ we already know that $S(X,Y;\Sigma_n)$ converge u.c.p.\@ as $n\to\infty$ to the stochastic process given by \eqref{eq:integraltoapprox}. This is the content of Theorem~21 in Chapter~II of the book \cite{MR2273672}.

In the general case, for any $\delta>0$ we find \cadlag{} martingales $X'=(X'_t)_{t\geq 0}$ and $Y'=(Y'_t)_{t\geq 0}$ with respect to $\mathcal{F}$ such that $\norm{X'}_{\LL^\infty}<\infty$, $\norm{X-X'}_{\LL^p}<\delta$, $\norm{Y'}_{\LL^2}<\infty$, and $\norm{Y-Y'}_{\LL^q}<\delta$. Rewrite the difference
$S_s(X,Y;\Sigma_m) - S_s(X,Y;\Sigma_n)$
as the sum of
\[ S_s(X',Y';\Sigma_m) - S_s(X',Y';\Sigma_n) \]
and
\begin{align}
& S_s(X-X',Y;\Sigma_m) + S_s(X',Y-Y';\Sigma_m) \nonumber \\
& + S_s(X',Y'-Y;\Sigma_n) + S_s(X'-X,Y;\Sigma_n). \label{eq:fourterms}
\end{align}
From the first part of the proof we know
\begin{equation}\label{eq:denseCauchy}
\lim_{m,n\to\infty} \mathbb{P}\Big( \sup_{s\in[0,t]} |S_s(X',Y';\Sigma_m) - S_s(X',Y';\Sigma_n)| > \varepsilon\Big) = 0
\end{equation}
for each $\varepsilon>0$ and each $t>0$.
By sampling arbitrary continuous-time martingales $\widetilde{X}$ and $\widetilde{Y}$ at times $t\wedge\tau^{(n)}_j$ we obtain discrete-time martingales such that $(S_{t\wedge\tau^{(n)}_j}(\widetilde{X},\widetilde{Y};\Sigma_n))_{j=0}^{l_n}$ is their paraproduct. Thus, estimate \eqref{eq:estChaoLongmax} applies and, together with Doob's inequality for $Y$, easily gives
\[ \norm[\Big]{\sup_{s\in[0,t]}\big|S_s(\widetilde{X},\widetilde{Y};\Sigma_n)\big|}_{\LL^r} \lesssim_{p,q} \norm[\big]{\widetilde{X}}_{\LL^p} \norm[\big]{\widetilde{Y}}_{\LL^q}, \]
with a constant independent of the partition $\Sigma_n$.
Applying this to each of the four terms in \eqref{eq:fourterms}, using the Markov--Chebyshev inequality, applying \eqref{eq:denseCauchy}, and finally letting $\delta\to0^+$, we obtain
\[ \limsup_{m,n\to\infty} \mathbb{P}\Big( \sup_{s\in[0,t]} |S_s(X,Y;\Sigma_m) - S_s(X,Y;\Sigma_n)| > \varepsilon\Big) = 0 \]
for $\varepsilon,t>0$.
Thus, $S(X,Y;\Sigma_n)$ converge u.c.p.\@ as $n\to\infty$ to some stochastic process, which we denote by $\Pi(X,Y)$. Note that $\Pi(X,Y)$ still has \cadlag{} paths a.s.\@, since this property is preserved under taking u.c.p.\@ limits. It is standard to conclude that $\Pi(X,Y)$ does not depend on the choice of $(\Sigma_n)_{n=0}^{\infty}$.

\eqref{it:cor:b} We explain how \eqref{eq:2variation} implies \eqref{eq:2variationcont}; very similarly one can use \eqref{eq:2jump} to prove \eqref{eq:2jumpcont}.
It is sufficient to establish a variant of \eqref{eq:2variationcont} in which the numbers $t_0,t_1,\ldots,t_m$ are only taken from a fixed finite set of nonnegative rational numbers $\Sigma$, but with a constant that does not depend on $\Sigma$. Afterwards, we can let those sets $\Sigma$ exhaust $[0,\infty)\cap\mathbb{Q}$, invoking the monotone convergence theorem. At the very end one can recall that $\Pi(X,Y)$ almost surely has \cadlag{} paths, so that $\Pi_{t,t'}(X,Y)$ is almost surely right-continuous in $t$ and $t'$.

Starting with a finite set $\Sigma$ we take an increasing sequence $(\Sigma_n)_{n=0}^{\infty}$ of finite subsets of $[0,\infty)$ with the following properties. If we write explicitly
\[ \Sigma_n = \Set[\big]{a^{(n)}_0, a^{(n)}_1, \ldots, a^{(n)}_{l_n}}, \quad a^{(n)}_0 < a^{(n)}_1 < \cdots < a^{(n)}_{l_n}, \]
then we require $\Sigma_0=\Sigma$, $a^{(n)}_0=0$ for $n\geq 1$, $\lim_{n\to\infty}a^{(n)}_{l_n}=\infty$, and
\[ \lim_{n\to\infty} \max_{1\leq j\leq l_n} \abs[\big]{a^{(n)}_j-a^{(n)}_{j-1}} = 0. \]
From part \eqref{it:cor:a} applied to deterministic partitions $\Sigma_n$ we know that
\begin{align*}
\Pi_{t_{k-1},t_k}(X,Y)
& = \lim_{n\to\infty} \big( S_{t_k}(X,Y;\Sigma_n) - S_{t_{k-1}}(X,Y;\Sigma_n) - X_{t_{k-1}} (Y_{t_k} - Y_{t_{k-1}}) \big) \\
& = \lim_{n\to\infty} \sum_{j\,:\,t_{k-1}<a^{(n)}_j\leq t_k} \big(X_{a^{(n)}_{j-1}} - X_{t_{k-1}}\big) \big(Y_{a^{(n)}_j} - Y_{a^{(n)}_{j-1}}\big)
\end{align*}
in probability. Repeatedly passing to almost surely convergent subsequences, we can assume that we already have almost sure convergence above for each of the finitely many choices of the numbers $t_0<t_1<\cdots<t_m$ from $\Sigma$ and for each $1\leq k\leq m$.
It remains to apply estimate \eqref{eq:2variation} to discrete-time martingales $(X_{a^{(n)}_j})_{j=0}^{l_n}$ and $(Y_{a^{(n)}_j})_{j=0}^{l_n}$ for each fixed $n\in\mathbb{N}$, recognizing their truncated paraproducts in the last display. Then we use Fatou's lemma as $n\to\infty$ to obtain control of the left hand side of \eqref{eq:2variationcont}.
\end{proof}

\section*{Acknowledgments}
V.K.\ was supported in part by the Croatian Science Foundation under the project UIP-2017-05-4129 (MUNHANAP).
P.Z.\ was partially supported by the Hausdorff Center for Mathematics (DFG EXC 2047) and DFG SFB 1060.
The authors also acknowledge support of the bilateral DAAD--MZO grant \emph{Multilinear singular integrals and applications}.
The authors would like to thank P.~Friz for turning their attention to the recent literature on martingale rough paths.

\printbibliography

\end{document}